\documentclass[12pt]{article}

\usepackage[english]{babel}
\usepackage{amsmath}
\usepackage{amssymb,amsfonts,amsthm}
\usepackage{tikz,color,pgf,graphicx,url}
\usetikzlibrary{calc}
\usetikzlibrary{decorations.pathreplacing}
\usepackage{enumerate}     

\usepackage{xcolor} 
\definecolor{mygreen}{RGB}{55, 184, 2}

\usepackage{hyperref}
\hypersetup{
	colorlinks   = true, 
	urlcolor     = black, 
	linkcolor    = black, 
	citecolor   = black 
}

\usepackage[
top=3cm,
bottom=3cm,
inner=2.5cm,      
outer=2.5cm,
footskip=40pt     
]{geometry}

\newtheorem{theorem}{Theorem}[section]
\newtheorem{corollary}[theorem]{Corollary}
\newtheorem{lemma}[theorem]{Lemma}

\newcommand{\gcg}{\gamma_{\rm cg}}

\newcommand{\gc}{\gamma_{\rm c}}

\newcommand{\F}{\mathcal{F}}

\newcommand{\fast}{{\sc Fast}}
\newcommand{\slow}{{\sc Slow}}

\usepackage[
backend=biber,
doi=false,
isbn=false,
url=false,
eprint=true,
sorting=nyt,
maxbibnames=999,
giveninits,
sortcites,]{biblatex}
\renewbibmacro{in:}{%
	\ifentrytype{article}{}{\printtext{\bibstring{in}\intitlepunct}}}
\AtEveryBibitem{%
	\clearfield{number}%
	\clearfield{note}%
	\clearfield{series}%
	\clearfield{pagetotal}}
\DeclareFieldFormat[article,inbook,incollection,inproceedings,patent,
thesis,unpublished]{title}{\mkbibemph{#1}\isdot}
\DeclareFieldFormat{journaltitle}{#1\isdot}
\DeclareFieldFormat{booktitle}{#1\isdot}
\DeclareFieldFormat[article]{pages}{#1}
\defcounter{biburlnumpenalty}{3000}
\begin{document}

\title{Complexity of the game connected domination problem}

\author{Vesna Ir\v si\v c Chenoweth\\ \texttt{vesna.irsic@fmf.uni-lj.si}}
\maketitle

\begin{center}
Faculty of Mathematics and Physics, University of Ljubljana, Slovenia\\
\medskip

Institute of Mathematics, Physics and Mechanics, Ljubljana, Slovenia\\
\medskip
\end{center}

\begin{abstract}
The connected domination game is a variant of the domination game where the played vertices must form a connected subgraph at all stages of the game. In this paper we prove that deciding whether the game connected domination number is smaller than a given integer is PSPACE-complete using log-space reductions for both Dominator- and Staller-start connected domination game.
\end{abstract}


\noindent
{\bf AMS Subj.\ Class.\ (2020):} 05C57, 05C69, 68Q17

\section{Introduction}
\label{sec:intro}

The \emph{connected domination game} was introduced in 2019 by Borowiecki, Fiedorowicz, and Sidorowicz~\cite{borowiecki+2019connected}, and has been studied afterwards in \cite{bujtas+2019connected, bujtas+2021connected, irsic2019+connected}. The game is played on a graph $G$ by two players: Dominator and Staller. They alternately select vertices of the graph. A move is legal if the selected vertex dominates at least one vertex which is not already dominated by previously played vertices, and if the set of vertices selected so far induces a connected subgraph of $G$. More precisely, choosing the vertex $v_i$ in the $i$th move is legal if for the vertices $v_1,\ldots,v_{i-1}$ chosen so far we have $N[v_i] \setminus \bigcup_{j=1}^{i-1}N[v_j]\not=\emptyset$, and the vertices $v_1, \ldots, v_i$ induce a connected subgraph of $G$.
Thus we only consider the connected domination game played on connected graphs. The game ends when there are no legal moves, so when the set of played vertices is a connected dominating set of $G$. The goal of Dominator is to finish the game with minimum number of moves, while the aim of Staller is to maximize the number of moves.

If both players play optimally, the number of moves on a graph $G$ is the \emph{game connected domination number} $\gcg(G)$ if Dominator starts the game on $G$ (D-game).  If Staller starts the game (S-game), the number of moves is the \emph{Staller-start game connected domination number} $\gcg'(G)$. Note that the terminology ``connected game domination number'' is also used in the literature, but we try to make it consistent with the terminology of other domination games (as in~\cite{bresar+2021book}). The moves in the D-game are denoted by $d_1, s_1, d_2, s_2, \ldots$, where $d_i$ are Dominator's moves and $s_i$ are Staller's moves, while the moves in the S-game are denoted by $s_1', d_1', \ldots$. For brevity, we sometimes refer to Dominator as he/him, and to Staller as she/her.

The relation between the Dominator- and Staller-start game connected domination number has been investigated in~\cite{borowiecki+2019connected, irsic2019+connected} and differs drastically from the analogous result for the (total) domination game. If $G$ is a connected graph, then $$\gcg(G) - 1 \leq \gcg'(G) \leq 2 \gcg(G)$$ and the bounds are tight. Recall that on the other hand, the difference between the Dominator- and Staller-start game (total) domination number is at most 1~\cite{bresar+2010original, henning+2015total, kinnersley+2013extremal}. 

The connected domination game is a variant of the classical domination game which was introduced by Brešar, Klavžar and Rall in~\cite{bresar+2010original}. For a survey of existing results (and the variants of the game) we refer the reader to~\cite{bresar+2021book}. The famous $\frac35$- and $\frac34$-conjectures have recently been resolved~\cite{portier-2022-3/4, versteegen-2023-3/5}, and progress towards solving the $\frac12$-conjecture has been made~\cite{portier-2024-1/2}. 
The domination game is known to be difficult from the algorithmic point of view; the decision version of the game is PSPACE-complete using log-space reductions~\cite{bresar+2016PSPACE2, klavzar+2015PSPACE1}. The game total domination problem is also PSPACE-complete using log-space reductions~\cite{bresar+2017pspace}. The connected variant of the coloring game has also been studied~\cite{charpentier+2020-con-col, lima+2023-con-col, bradshaw2023-con-col}.

A useful tool for proving bounds for the domination games is the so-called imagination strategy and we will use it several times in this paper. It was first used in~\cite{bresar+2010original}; see also~\cite[Section 2.2]{bresar+2021book}. The general idea of the strategy is that while a domination game is played on a graph, one of the players imagines another game is played in parallel on a modified graph. The player imagining the game plays optimally in the imagined game and copies their moves to the real game, while they also copy the moves of the other player from the real to the imagined game. Care needs to be taken to make this copied moves legal. In our case, we will also use the POS-CNF game as the imagined game, thus we will need to describe even more precisely how the moves are copied between the games.

In this paper we prove that the connected domination game is of the same complexity. In Section \ref{sec:prelim}, we recall known results that are needed in the rest of the paper. The proof that the game connected domination problem is log-complete in PSPACE is given in Section~\ref{sec:D-game}, while an analogous result for the Staller-start game is presented in Section~\ref{sec:S-game}.

\section{Preliminaries}
\label{sec:prelim}

For an integer $n$, let $[n] = \{1,\ldots,n\}$.
Let $G$ be a graph. A set $S \subseteq V(G)$ is a \emph{connected dominating set} of $G$ if every vertex in $V(G) \setminus S$ has a neighbor in $S$ and the subgraph of $G$ induced on $S$ is connected. The smallest size of a connected dominating set is the \emph{connected domination number} $\gc(G)$ of $G$. Each connected dominating set of $G$ of size $\gc(G)$ is called a \emph{$\gc$-set} of $G$. We also recall the following result.

\begin{theorem}[{\cite[Theorem 1]{borowiecki+2019connected}}]
	\label{thm:basic}
	If $G$ is a connected graph, then $\gc(G) \leq \gcg(G) \leq 2 \gc(G) - 1$.
\end{theorem}

A decision problem is \emph{PSPACE-complete} if it can be solved using working space of polynomial size with respect to the input length, and every other problem solvable in polynomial space can be reduced to it in polynomial time. If a problem is PSPACE-complete using log-space reductions, we say that it is \emph{log-complete in PSPACE}.

We consider the following decision problems.

\begin{center}
	\begin{tabular}{l}
		\textsc{Game connected domination problem}\\
		Input: A graph $G$ and an integer $m$.\\ \vspace{0.5cm}
		Question: Is $\gcg(G) \leq m$?\\
		
		\textsc{Staller-start game connected domination problem}\\
		Input: A graph $G$ and an integer $m$.\\ \vspace{0.5cm}
		Question: Is $\gcg'(G) \leq m$?\\
		
		\textsc{POS-CNF problem}\\
		Input: A positive CNF formula $\F$ with $k$ variables and $n$ clauses.\\
		Question: Does Player 1 win on $\F$?
	\end{tabular}
\end{center}

The \textsc{POS-CNF problem} is known to be log-complete in PSPACE~\cite{pos-cnf}. In the POS-CNF game, we are given a formula $\F$ with $k$ variables that is a conjunction of $n$ disjunctive clauses in which only positive variables appear. Two players alternate turns, Player 1 setting a previously unset variable TRUE, and Player 2 setting one FALSE. When all $k$ variables are set, Player 1 wins if the formula $\F$ is TRUE, otherwise Player 2 wins.

Recall the properties of the following graph from \cite{bujtas+2021connected, irsic2019+connected}. 
Let $H_n$, $n \geq 2$, be a graph with vertices $V(G_n) = \{ u_0, \ldots, u_{n+1}\} \cup \{ x_1, \ldots, x_{n-1}\} \cup \{y_1,  \ldots, y_{n-1} \}$ and edges $u_i u_{i+1}$ for $i \in \{0, \ldots, n\}$, $u_i x_i$, $x_i y_i$, $y_i u_{i+1}$, and $u_{i+1}  x_i$ for $i \in [n-1]$. For example, see Figure~\ref{fig:primer-H6}.

\begin{figure}[!ht]
	\begin{center}
		\begin{tikzpicture}[thick, scale=1.2pt]
			\tikzstyle{every node}=[circle, draw, fill=black!10, inner sep=0pt, minimum width=4pt]

			\node[label=below: {$u_0$}] (u0) at (0,0) {};
			\node[label=135: {$u_1$}] (u1) at (1,0) {};
			\node[label=-135: {$u_2$}] (u2) at (2,0) {};
			\node[label=135: {$u_3$}] (u3) at (3,0) {};
			\node[label=-135: {$u_4$}] (u4) at (4,0) {};
			\node[label=135: {$u_5$}] (u5) at (5,0) {};
			\node[label=below: {$u_6$}] (u6) at (6,0) {};
			\node[label=below: {$u_7$}] (u7) at (7,0) {};
			
			\node[label=above: {$x_1$}] (x1) at (1,1) {};
			\node[label=below: {$x_2$}] (x2) at (2,-1) {};
			\node[label=above: {$x_3$}] (x3) at (3,1) {};
			\node[label=below: {$x_4$}] (x4) at (4,-1) {};
			\node[label=above: {$x_5$}] (x5) at (5,1) {};
			
			\node[label=above: {$y_1$}] (y1) at (2,1) {};
			\node[label=below: {$y_2$}] (y2) at (3,-1) {};
			\node[label=above: {$y_3$}] (y3) at (4,1) {};
			\node[label=below: {$y_4$}] (y4) at (5,-1) {};
			\node[label=above: {$y_5$}] (y5) at (6,1) {};

			\draw (u0) -- (u1) -- (u2) -- (u3) -- (u4) -- (u5) -- (u6);
			
			\draw (u1) -- (x1) -- (y1);
			\path (u2) edge (x1);
			\path (u2) edge (y1);
			
			\draw (u2) -- (x2) -- (y2);
			\path (u3) edge (x2);
			\path (u3) edge (y2);
			
			\draw (u3) -- (x3) -- (y3);
			\path (u4) edge (x3);
			\path (u4) edge (y3);
			
			\draw (u4) -- (x4) -- (y4);
			\path (u5) edge (x4);
			\path (u5) edge (y4);
			
			\draw (u5) -- (x5) -- (y5);
			\path (u6) edge (x5);
			\path (u6) edge (y5);
			
			\path (u6) edge (u7);
			
		\end{tikzpicture}
		\caption{The graph $H_6$.}
		\label{fig:primer-H6}
	\end{center}
\end{figure}
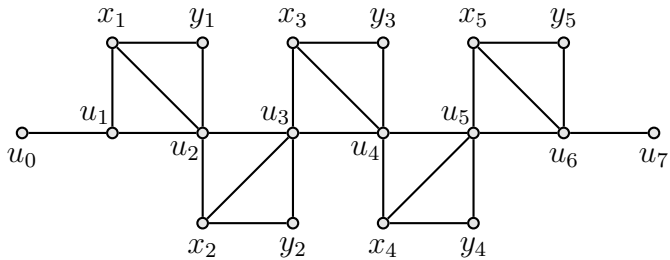

\begin{lemma}[{\cite{bujtas+2021connected}}]
	\label{lem:Hn}
	If $n \geq 2$, then $\gcg(H_n) = n$ and $\gcg'(H_n) = 2n$.
\end{lemma}

Dominator's strategy for the D-game on $H_n$ is to play $d_1 = u_n$ which makes all the remaining moves unique and the game finishes in $n$ moves. As in~\cite{bujtas+2021connected}, we call this strategy \fast. Note that exactly vertices $u_n, \ldots, u_1$ are played during the game. 

Staller's strategy in the S-game is to start on $u_0$ and to play vertices $x_1, \ldots, x_{n-1}$ whenever she can. She is able to force $n-1$ additional moves on $V(H_n) \setminus \{ u_0, \ldots, u_{n+1} \}$. Thus, apart from the vertices $u_1, \ldots, u_n$, exactly $n$ additional moves are played (counting the move $u_0$ as well). We call this strategy of Staller \slow. 

Observe that since $\gamma_{\rm c}(H_n) = n$, there are always at least $n$ moves played on $H_n$ (even if players do not alternate taking moves). It also follows from the above that $n$ moves are played on $H_n$ (again, even if players do not alternate taking moves) if and only if exactly vertices $u_1, \ldots, u_n$ are played.

\section{Complexity of the Dominator-start game}
\label{sec:D-game}

In this section we provide a reduction from the \textsc{POS-CNF problem} to the \textsc{Game connected domination problem}. The construction needed is described in Section~\ref{sec:construction}, while the properties of the reduction and the final results are presented in Section~\ref{sec:proof}. Note that the complexities of the domination and total domination games are also determined using a reduction from the \textsc{POS-CNF problem}, but the constructions used are different for each game; see~\cite{bresar+2016PSPACE2, bresar+2017pspace}.

\subsection{Construction}
\label{sec:construction}

Let the graph $B$ be as in Figure \ref{fig:gadget}. When we refer to a copy $B_i$ of the graph $B$, its vertices are labeled $w_i$ for every $w \in \{a,e,b,b',h,k, f^1,g^1,f^2,g^2\}$.

\begin{figure}[!ht]
	\begin{center}
		\begin{tikzpicture}[thick, scale = 1.5]
			\tikzstyle{every node}=[circle, draw, fill=black!10, inner sep=0pt, minimum width=4pt]
			
			\node[label=left: {$a$}] (a) at (0,0) {};
			\node[label=below left: {$e$}] (e) at (1,0) {};
			\node[label=below right: {$b$}] (b) at (2,0) {};
			\node[label=right: {$b'$}] (d) at (3,0) {};
			
			\node[label=above: {$h$}] (h) at (1,1) {};
			\node[label=above left: {$k$}] (k1) at (0.5,1) {};
			
			\node[label=above: {$f^1$}] (f1) at (2,1) {};
			\node[label=above: {$g^1$}] (g1) at (1.5,1) {};
			
			\node[label=below: {$f^2$}] (f2) at (2,-1) {};
			\node[label=below: {$g^2$}] (g2) at (1.5,-1) {};

			\draw (a) -- (e) -- (b) -- (d);
			\draw (a) -- (h) -- (e) -- (f1) -- (b) -- (f2) -- (e) -- (g1) -- (f1);
			\draw (e) -- (g2) -- (f2);
			\draw (h) -- (k1) -- (a);			
		\end{tikzpicture}
		\caption{The graph $B$.}
		\label{fig:gadget}
	\end{center}
\end{figure}
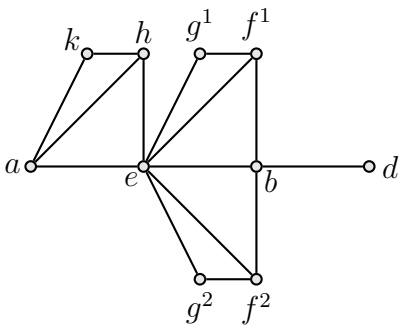

\begin{lemma}
	\label{lem:gadget}
	If $B$ is the graph defined above, then $\gc(B) = 3$. The same holds even if vertices $a$ and $b$ are predominated.
\end{lemma}

For $m \geq 1$, let the graph $C(m)$ have the vertex set $\{ c, c^1, \ldots, c^m, d, d^1, \ldots, d^m \}$ and edges $cd$, $c d^i$ for every $i \in [m]$ and $c^i d^i$ for every $i \in [m]$. For example, see Figure \ref{fig:gadget-C}. When we refer to a copy $C(m)_i$ of the graph $C(m)$, its vertices are labeled as $w_i$ for every $w \in V(C(m))$. When $m$ is clear from the context, we simply write $C$ instead of $C(m)$.

\begin{figure}[!ht]
	\begin{center}
		\begin{tikzpicture}[thick, scale = 1.5]
			\tikzstyle{every node}=[circle, draw, fill=black!10, inner sep=0pt, minimum width=4pt]
			
			\node[label=above: {$c$}] (c) at (0,0) {};
			\node[label=above: {$c^1$}] (c1) at (1,0) {};
			\node[label=above: {$c^2$}] (c2) at (2,0) {};
			\node[label=above: {$c^3$}] (c3) at (3,0) {};
			
			\node[label=below: {$d$}] (d) at (0,-1) {};
			\node[label=below: {$d^1$}] (d1) at (1,-1) {};
			\node[label=below: {$d^2$}] (d2) at (2,-1) {};
			\node[label=below: {$d^3$}] (d3) at (3,-1) {};
						
			\draw (c) -- (d);
			\draw (c1) -- (d1);
			\draw (c2) -- (d2);
			\draw (c3) -- (d3);
			
			\draw (c) -- (d1);	
			\draw (c) -- (d2);
			\draw (c) -- (d3);	
		\end{tikzpicture}
		\caption{The graph $C(3)$.}
		\label{fig:gadget-C}
	\end{center}
\end{figure}

Let $A$ be the graph in Figure \ref{fig:primer-A}. Observe that if the first move played on $A$ is $p_1$ then four moves are needed to end the game if Dominator starts and three moves if Staller starts.

\begin{figure}[!ht]
	\begin{center}
		\begin{tikzpicture}[thick, scale=1.2pt]
			\tikzstyle{every node}=[circle, draw, fill=black!10, inner sep=0pt, minimum width=4pt]
			
			\node[label=135: {$p_1$}] (u1) at (1,0) {};
			\node[label=-135: {$p_2$}] (u2) at (2,0) {};
			\node[label=135: {$p_3$}] (u3) at (3,0) {};
						
			\node[label=above: {$q_1$}] (x1) at (1,1) {};
			\node[label=below: {$q_2$}] (x2) at (2,-1) {};
			
			\node[label=above: {$r_1$}] (y1) at (2,1) {};
			\node[label=below: {$r_2$}] (y2) at (3,-1) {};
						
			\draw (u1) -- (u2) -- (u3);
			
			\draw (u1) -- (x1) -- (y1);
			\path (u2) edge (x1);
			\path (u2) edge (y1);
			
			\draw (u2) -- (x2) -- (y2);
			\path (u3) edge (x2);
			\path (u3) edge (y2);
			
		\end{tikzpicture}
		\caption{The graph $A$.}
		\label{fig:primer-A}
	\end{center}
\end{figure}

Given a formula $\F$ with $k$ variables and $n$ disjunctive clauses, we built a graph $G_{\F}$ in the following way. For each variable $X_i$, $i \in [k]$, we add to the graph a copy $B_i$ of the graph $B$ (called a gadget $B_i$). For each clause $C_j$, $j \in [n]$, we add to the graph a copy $C(n)_j = C_j$ of the graph $C(n)$ (called a gadget $C_j$) and make all vertices $\{c_j, c^1_j, \ldots, c^n_j\}$ adjacent to $a_i$ if and only if the variable $X_i$ appears in the clause $C_j$. We add to the graph a disjoint copy of the graph $H_{2n+7}$ and make $u_0$ adjacent to vertices $a_i$ and $b_i$ for all $i \in [k]$. If $k$ is odd, we also add to the graph a disjoint copy of the graph $A$ and make $p_1$ adjacent to $u_0$. For example, see Figure \ref{fig:GF}.

\begin{figure}[!ht]
	\begin{center}
		\begin{tikzpicture}[thick, scale = 0.77]
			\tikzstyle{vert}=[circle, draw, fill=black!10, inner sep=0pt, minimum width=4pt]
			
			\foreach \x in {1,2,3,4,5}{
			\node (k\x) at (4*\x - 4 +1.5,1.5) {$B_{\x}$};
			\node[vert, label=below: {$a_{\x}$}] (a\x) at (4*\x - 4 +0,0) {};
			\node[vert] (e\x) at (4*\x - 4+1,0) {};
			\node[vert, label=below right: {$b_{\x}$}] (b\x) at (4*\x - 4+2,0) {};
			\node[vert] (d\x) at (4*\x - 4+2.6,0.3) {};
			
			\node[vert] (h\x) at (4*\x - 4+1,1) {};
			\node[vert] (k1\x) at (4*\x - 4+0.5,1) {};
			
			\node[vert] (f1\x) at (4*\x - 4+2,1) {};
			\node[vert] (g1\x) at (4*\x - 4+1.5,1) {};
			
			\node[vert] (f2\x) at (4*\x - 4+2,-1) {};
			\node[vert] (g2\x) at (4*\x - 4+1.5,-1) {};

			\draw (a\x) -- (e\x) -- (b\x) -- (d\x);
			\draw (a\x) -- (h\x) -- (e\x) -- (f1\x) -- (b\x) -- (f2\x) -- (e\x) -- (g1\x) -- (f1\x);
			\draw (e\x) -- (g2\x) -- (f2\x);
			\draw (h\x) -- (k1\x) -- (a\x);	
			}
		
		\node (h) at (14,8.5) {$H_{13}$};
		\node[vert, label=5: {$u_0$}] (u0) at (9.5,5) {};
		\node[vert, label=180: {$u_1$}] (u1) at (9.5,7) {};
		\node[vert] (u2) at (10.5,7) {};
		\node[vert] (u3) at (11.5,7) {};
		\node[vert] (u4) at (12.5,7) {};
		\node[vert, label=90: {$u_5$}] (u5) at (13.5,7) {};
		\node (ddd) at (14.5,7) {$\cdots$};
		\node[vert, label=90: {$u_{11}$}] (u8) at (15.5,7) {};
		\node[vert] (u9) at (16.5,7) {};
		\node[vert, label=-90: {$u_{13}$}] (u10) at (17.5,7) {};
		\node[vert, label=-90: {$u_{14}$}] (u11) at (18.5,7) {};
		
		\node[vert] (x1) at (9.5,8) {};
		\node[vert] (x2) at (10.5,6) {};
		\node[vert] (x3) at (11.5,8) {};
		\node[vert] (x4) at (12.5,6) {};
		\node[vert] (x8) at (15.5,6) {};
		\node[vert] (x9) at (16.5,8) {};

		\node[vert] (y1) at (10.5,8) {};
		\node[vert] (y2) at (11.5,6) {};
		\node[vert] (y3) at (12.5,8) {};
		\node[vert] (y4) at (13.5,6) {};
		\node[vert] (y8) at (16.5,6) {};
		\node[vert] (y9) at (17.5,8) {};
		
		\draw (u5) -- (14,7);
		\draw (u8) -- (15,7);
		
		\draw (u0) -- (u1) -- (u2) -- (u3) -- (u4) -- (u5);
		\draw  (u8) -- (u9) -- (u10) -- (u11);
		
		\draw (u1) -- (x1) -- (y1);
		\path (u2) edge (x1);
		\path (u2) edge (y1);
		
		\draw (u2) -- (x2) -- (y2);
		\path (u3) edge (x2);
		\path (u3) edge (y2);
		
		\draw (u3) -- (x3) -- (y3);
		\path (u4) edge (x3);
		\path (u4) edge (y3);
		
		\draw (u4) -- (x4) -- (y4);
		\path (u5) edge (x4);
		\path (u5) edge (y4);
		
		\draw (u8) -- (x8) -- (y8);
		\path (u9) edge (x8);
		\path (u9) edge (y8);
		
		\draw (u9) -- (x9) -- (y9);
		\path (u10) edge (x9);
		\path (u10) edge (y9);
		
		\path (u10) edge (u11);
		
		\draw (u0) to[out=180,in=135,distance=3cm] (a1);
		\draw (u0) to[out=-160,in=60,distance=3cm] (b1);
		\draw (u0) to[out=-140,in=120,distance=3cm] (a2);
		\draw (u0) to[out=-120,in=60,distance=3cm] (b2);
		
		\draw (u0) to[out=-100,in=120,distance=2cm] (a3);
		\draw (u0) to[out=-80,in=60,distance=2cm] (b3);
		
		\draw (u0) to[out=-60,in=120,distance=3cm] (a4);
		\draw (u0) to[out=-40,in=60,distance=3cm] (b4);
		\draw (u0) to[out=-20,in=120,distance=3cm] (a5);
		\draw (u0) to[out=0,in=45,distance=3cm] (b5);
		
		\node (a) at (2.5,7) {$A$ as $k$ is odd};		
		\node[vert, label=45: {$p_1$}] (p1) at (6.5,7) {};
		\node[vert, label=-45: {$p_2$}] (p2) at (5.5,7) {};
		\node[vert, label=90: {$p_3$}] (p3) at (4.5,7) {};
		
		\node[vert, label=above: {$q_1$}] (q1) at (6.5,8) {};
		\node[vert, label=below: {$q_2$}] (q2) at (5.5,6) {};
		
		\node[vert, label=above: {$r_1$}] (r1) at (5.5,8) {};
		\node[vert, label=below: {$r_2$}] (r2) at (4.5,6) {};
		
		\draw (p1) -- (p2) -- (p3);
		
		\draw (p1) -- (q1) -- (r1);
		\path (p2) edge (q1);
		\path (p2) edge (r1);
		
		\draw (p2) -- (q2) -- (r2);
		\path (p3) edge (q2);
		\path (p3) edge (r2);
		
		\draw (u0) -- (p1);
		
		\node (cc1) at (2,-6.3) {$C_1$};
		\node[vert, label=above: {$c_1$}] (c1) at (0.5,-4) {};
		\node[vert, label=right: {$c^1_1$}] (c11) at (1.5,-4) {};
		\node[vert, label=right: {$c^2_1$}] (c21) at (2.5,-4) {};
		\node[vert, label=right: {$c^3_1$}] (c31) at (3.5,-4) {};
		
		\node[vert, label=below: {$d_1$}] (d1) at (0.5,-5) {};
		\node[vert, label=below: {$d^1_1$}] (d11) at (1.5,-5) {};
		\node[vert, label=below: {$d^2_1$}] (d21) at (2.5,-5) {};
		\node[vert, label=below: {$d^3_1$}] (d31) at (3.5,-5) {};
		
		\draw (c1) -- (d1);
		\draw (c11) -- (d11);
		\draw (c21) -- (d21);
		\draw (c31) -- (d31);
		
		\draw (c1) -- (d11);	
		\draw (c1) -- (d21);
		\draw (c1) -- (d31);
		
		\node (cc2) at (8.5,-6.3) {$C_2$};
		\node[vert, label=above: {$c_2$}] (c2) at (7,-4) {};
		\node[vert, label=right: {$c^1_2$}] (c12) at (8,-4) {};
		\node[vert, label=right: {$c^2_2$}] (c22) at (9,-4) {};
		\node[vert, label=right: {$c^3_2$}] (c32) at (10,-4) {};
		
		\node[vert, label=below: {$d_2$}] (d2) at (7,-5) {};
		\node[vert, label=below: {$d^1_2$}] (d12) at (8,-5) {};
		\node[vert, label=below: {$d^2_2$}] (d22) at (9,-5) {};
		\node[vert, label=below: {$d^3_2$}] (d32) at (10,-5) {};
		
		\draw (c2) -- (d2);
		\draw (c12) -- (d12);
		\draw (c22) -- (d22);
		\draw (c32) -- (d32);
		
		\draw (c2) -- (d12);	
		\draw (c2) -- (d22);
		\draw (c2) -- (d32);
		
		\node (cc3) at (15,-6.3) {$C_3$};
		\node[vert, label=above: {$c_3$}] (c3) at (13.5,-4) {};
		\node[vert, label=right: {$c^1_3$}] (c13) at (14.5,-4) {};
		\node[vert, label=right: {$c^2_3$}] (c23) at (15.5,-4) {};
		\node[vert, label=right: {$c^3_3$}] (c33) at (16.5,-4) {};
		
		\node[vert, label=below: {$d_3$}] (d3) at (13.5,-5) {};
		\node[vert, label=below: {$d^1_3$}] (d13) at (14.5,-5) {};
		\node[vert, label=below: {$d^2_3$}] (d23) at (15.5,-5) {};
		\node[vert, label=below: {$d^3_3$}] (d33) at (16.5,-5) {};
	
		\draw (c3) -- (d3);
		\draw (c13) -- (d13);
		\draw (c23) -- (d23);
		\draw (c33) -- (d33);
		
		\draw (c3) -- (d13);	
		\draw (c3) -- (d23);
		\draw (c3) -- (d33);
		
		\draw (c1) to[out=135,in=-135,distance=2cm] (a1);
		\draw (c11) to[out=135,in=-135,distance=2cm] (a1);
		\draw (c21) to[out=130,in=-135,distance=2cm] (a1);
		\draw (c31) to[out=130,in=-135,distance=2cm] (a1);
		
		\draw (c1) to[out=55,in=-135,distance=3cm] (a3);
		\draw (c11) to[out=60,in=-135,distance=3cm] (a3);
		\draw (c21) to[out=65,in=-135,distance=3cm] (a3);
		\draw (c31) to[out=70,in=-135,distance=3cm] (a3);
		
		\draw (c2) to[out=135,in=-145,distance=2.5cm] (a2);
		\draw (c12) to[out=135,in=-145,distance=2.5cm] (a2);
		\draw (c22) to[out=130,in=-145,distance=2.5cm] (a2);
		\draw (c32) to[out=130,in=-145,distance=2.5cm] (a2);
		
		\draw (c2) to[out=50,in=-45,distance=1.3cm] (a3);
		\draw (c12) to[out=60,in=-45,distance=1.3cm] (a3);
		\draw (c22) to[out=70,in=-45,distance=1.3cm] (a3);
		\draw (c32) to[out=80,in=-45,distance=1.3cm] (a3);
		
		\draw (c2) to[out=45,in=-135,distance=4cm] (a5);
		\draw (c12) to[out=45,in=-135,distance=4cm] (a5);
		\draw (c22) to[out=45,in=-135,distance=4cm] (a5);
		\draw (c32) to[out=45,in=-135,distance=4cm] (a5);
		
		\draw (c3) to[out=135,in=-135,distance=2.5cm] (a4);
		\draw (c13) to[out=135,in=-135,distance=2.5cm] (a4);
		\draw (c23) to[out=130,in=-135,distance=2.5cm] (a4);
		\draw (c33) to[out=130,in=-135,distance=2.5cm] (a4);
		
		\draw (c3) to[out=45,in=-45,distance=2cm] (a5);
		\draw (c13) to[out=55,in=-45,distance=2cm] (a5);
		\draw (c23) to[out=65,in=-45,distance=2cm] (a5);
		\draw (c33) to[out=75,in=-45,distance=2cm] (a5);
		\end{tikzpicture}
		\caption{The graph $G_\F$ obtained from the formula $\F = (X_1 \vee X_3) \wedge (X_2 \vee X_3 \vee X_5) \wedge (X_4 \vee X_5)$.}
		\label{fig:GF}
	\end{center}
\end{figure}

\begin{lemma}
	\label{lem:gc-min}
	If $\F$ is a forumla with $k$ variables and $n$ disjunctive clauses, and $G_\F$ is the graph defined above, then $$\gc(G_\F) = \begin{cases}
		3k + 3n + 8; & k \text{ even},\\
		3k + 3n + 11; & k \text{ odd},
	\end{cases}$$ and every $\gc$-set of $G_{\F}$ contains vertices $\{u_0, u_1, \ldots, u_{2n+7}\}$, vertices $\{c_1, \ldots, c_n\}$, exactly three vertices from each gadget $B_i$, $i \in [k]$, and if $k$ is odd also exactly three vertices from $A$.
\end{lemma}

\begin{proof}
	Let $D = \{u_0, u_1, \ldots, u_{2n+7}\} \cup \bigcup_{i\in [k]} \{a_i, e_i, b_i\} \cup \bigcup_{j \in [n]} \{c_j\}$. Clearly, $|D| = 3k+3n+8$ and if $k$ is even, then $D$ is a connected dominating set of $G_{\F}$. If $k$ is odd, $D \cup \{p_1, p_2, p_3\}$ is of order $3k+3n+11$ and is a connected dominating set of $G_{\F}$. Thus the desired upper bound for $\gc(G_\F)$ follows.
	
	By Lemma \ref{lem:gadget}, three vertices are needed to dominate each gadget $B_i$ (even if $a_i$ and $b_i$ are already dominated). At least one vertex from each $C_j$, $j \in [n]$, is needed to dominate $C_j$, and $c_j$ is the unique such vertex. By Lemma \ref{lem:Hn}, $2n+7$ vertices are needed to dominate $H_{2n+7}$ and the unique $\gc$-set of $H_{2n+7}$ of size $2n+7$ is $\{u_1, \ldots, u_{2n+7}\}$. But since the connected dominating set of $G_\F$ must be connected, $u_0$ must be included in any connected dominating set of $G_\F$. If $k$ is odd, at least three vertices are needed to dominate $A$. Thus the desired lower bound for $\gc(G_\F)$ follows, as does the description of a minimum connected dominating set.
\end{proof}

The following simple observations about the connected domination game on $G_{\F}$ that easily follow from Lemma~\ref{lem:gc-min} will also be useful:
\begin{itemize}
	\item At least $2n+8$ moves must be played in the copy of $H_{2n+7}$ in $G_\F$, including the vertex $u_0$.
	\item At least three moves must be played in each gadget $B_i$.
	\item If Staller is the first to play on $C_j$, then at least two moves are played on $C_j$.
\end{itemize}

\subsection{Proof of the PSPACE-completeness}
\label{sec:proof}

We consider two auxiliary lemmas that are needed for the main result.

\begin{lemma}
	\label{lem:p1-wins}
	If Player 1 has a winning strategy for the POS-CNF game played on $\F$, then $$\gcg(G_\F) \leq \begin{cases}
		3k + 4n + 8; & k \text{ even},\\
		3k + 4n + 11; & k \text{ odd}.
	\end{cases}$$
\end{lemma}

\begin{proof}
	We describe Dominator's strategy that ensures that the game ends in at most $3k + 4n + 8$ moves if $k$ is even and in at most $3k + 4n + 11$ moves if $k$ is odd. Most of the proof is the same for both parities of $k$.
	
	Dominator starts the game by playing $d_1 = u_{2n+7}$. Thus he forces the strategy \fast\ to be played on $H_{2n+7}$, so the moves $s_1 = u_{2n+6}, d_2 = u_{2n+5}, \ldots, d_{2n+7} = u_1, s_{2n+8} = u_0$ are forced. As $2n+8$ is even, Staller plays $u_0$. Observe that no vertex $d_j^m$ can be a legal move in the rest of the game.
	
	Dominator's next move is to play $d_{2n+9} = a_i$ if setting $X_i$ to TRUE is the optimal first move of Player 1 in the POS-CNF game on $\F$. For the rest of the game, Dominator will imagine a game is being played on $\F$, specifying certain moves of Staller and himself as moves in the POS-CNF game. More precisely, in the game on $\F$, Player 1 will be playing optimally, and their moves will determine some of the corresponding moves of Dominator on $G_\F$. As Player 1 has a winning strategy on $\F$, they can win no matter how Player 2 plays. Player 2's moves on $\F$ will be determined by some of Staller's moves on $G_\F$. Which moves of Staller mean that Player 2 makes a move on $\F$ and how does Player 1's reply on $\F$ translate to Dominator's reply on $G_\F$ is explained in (6) below.
	
	We will prove that Dominator can ensure on average at most three moves on each $B_i$, at most two moves on each $C_j$ and at most three moves on $A$. By saying on average we mean that the moves can be counted as at most three on each $B_i$ and at most two on each $C_j$, as Dominator's strategy ensures that if four moves are played on some $B_i$, there is a $C_j$ with only one move played on it. 
	
	Dominator follows the rules listed below (each rule is roughly summarized first, followed by the exact Dominator's strategy) with the following addition: when we say that Dominator plays \emph{any legal move}, we mean that he plays any $c_j$, $a_i$ or $e_i$ if possible (in this order of preference), and otherwise he plays any vertex. This is to ensure at most two moves on each $C_j$ and at most three moves on each $B_i$ even when Dominator does not have a more precise strategy. As a result, it might sometimes happen that Dominator's strategy below wants him to play a vertex that is not a legal move anymore. In this case, he plays any legal move (with the rules given above), but he still performs potential other actions associated with the move that is not legal anymore as if he played it now. It is thus necessary for Dominator to keep track of the vertices he played due to an exact strategy and the vertices he played as any legal moves. 
	\begin{enumerate}[(1)]
		\item If Staller plays $p_1$, Dominator replies by playing $p_2$.\\
		This ensures that if Staller is the first to play on $A$, then exactly three moves are made on $A$ during the game.
		
		\item If Staller plays $p_3$ or $q_2$, then Dominator plays any legal move.\\
		Note that this means that on $G - H_{2n+7} -A$, Dominator makes two consecutive moves which we have to consider when analyzing some of the cases below.
		
		\item If Staller plays on $C_j$, Dominator plays $c_j$.\\
		If Dominator played first on $C_j$ (before this move of Staller), then his strategy ensures he played $c_j$, thus no more moves are legal on $C_j$. So we know that the current move of Staller is in fact the first move on $C_j$. If there are still undominated vertices on $C_j$ after Staller's move, $c_j$ is a legal reply for Dominator that ensures that at most two moves are made on $C_j$ during the game. Otherwise, so if Staller's move on $C_j$ leaves no undominated vertices in $C_j$, then Staller played $c_j$, Dominator plays any legal move, and just one move is made on $C_j$ during the game.
		
		\item If Staller plays on $B_i$ and not all vertices of $B_i$ are dominated after her move, Dominator replies on $B_i$ as well.\\
		More precisely, if Staller is the first to play on some $B_i$, then Dominator replies by playing $e_i$. If Staller is not the first to play on $B_i$, but $B_i$ is still not entirely dominated, then this is only possible if Dominator played the first move on this $B_i$, so by (6), he played $a_i$ before. If Staller now plays $e_i$, he replies on $b_i$, and if she plays $b_i$, he selects $e_i$ (no other move on $B_i$ is a legal move for Staller at this stage of the game). Note that if Dominator made two consecutive moves on $B_i$, they were $a_i$ and $e_i$, thus after Staller's move on $B_i$, there are no more undominated vertices left and we are not in this case. Similarly if Dominator made three consecutive moves on $B_i$.
		
		Notice that with this strategy, Dominator ensures that no more than three moves are played on each $B_i$ unless Staller plays a move from (5).
		
		\item If Staller plays $a_i$ and all vertices of $B_i$ are dominated already before her move, Dominator replies by playing $c_j$ in the appropriate gadget $C_j$.\\
		Note that if such a move of Staller was legal, it must have newly dominated vertices $c_j, c_j^1, \ldots, c_j^n$ for some $j \in [n]$. Thus Dominator's reply on $c_j$ is a legal move. Additionally, notice that while four moves were played on $B_i$, only one move was played on $C_j$ (and no more moves on $B_i$ or $C_j$ are possible during the game). For counting purposes, we consider this as three moves on $B_i$ and two on $C_j$.
		
		\item If Staller plays on $B_i$ that was not dominated before her move and all vertices of $B_i$ are dominated after her move, Dominator selects an appropriate gadget $B_{j}$ to play on.\\
		If $X_i$ already has an assigned value in the imagined POS-CNF game on $\F$, then Dominator plays any legal move. Otherwise, Dominator imagines that Player 2 set $X_i$ to FALSE in the POS-CNF game on $\F$ now. If this ends the POS-CNF game on $\F$, Dominator plays any legal move. Otherwise, Dominator considers Player 1's optimal strategy in the POS-CNF game, and plays on $B_{j}$ such that Player 1's strategy is to set $X_{j}$ to TRUE. Dominator's strategy is to play $a_j$ if possible, otherwise he plays $b_j$ if this is the last move on $B_j$, or he plays any legal move (in this order of preference). Once Dominator sets $X_j$ to TRUE as Player 1's move on $\F$ and wants to play on $B_j$, the following options are possible for $B_j$: no move has been played on it so far (Dominator now plays $a_j$), Dominator already played one or more moves on it (Dominator now plays $b_j$ if this is the last move on $B_j$, or he plays any legal move, but we know he played $a_j$ before), Staller and Dominator both played on $B_j$ before (Dominator now plays $a_j$ if possible, or $b_j$ if this is the last move on $B_j$, or he plays any legal move, but again we know that $a_j$ has been played before). In all cases we notice that Dominator is ensuring that at most three moves have been or will be played on $B_j$, $a_j$ is one of them, and that if Staller was the first to play on $B_j$, all vertices of $B_j$ are dominated after this move.

		Notice that with this strategy, Dominator ensures that for all indices $i \in I$ that Player 1 sets to TRUE in his winning strategy in the POS-CNF game on $\F$, the vertices $a_i$, $i \in I$, are played. Thus by the end of the POS-CNF game, all vertices $c_1, \ldots, c_n, \ldots, c_1^n, \ldots, c_n^n$ will be dominated in $G_\F$ (since Player 1 wins on $\F$) and so all vertices from $C_1, \ldots, C_n$ can be dominated by using at most three moves on each $B_i$ and at most two on each $C_j$ (on average).

	\end{enumerate}
	With the described strategy, Dominator ensures exactly $2n+8$ moves on $H_{2n+7}$, and on average at most three moves on each $B_i$ and at most two moves on each $C_j$. Observe also that the imagined POS-CNF game either ends with the win of Player 1 (meaning that all vertices $c_1, \ldots, c_n, \ldots, c_1^n, \ldots, c_n^n$ are dominated and are thus legal moves can be played on any $C_1, \ldots, C_n$ that is not dominated yet) or the POS-CNF game does not end as some variables never get assigned values. By (6) we know that Dominator played the third and last move on all the remaining (unassigned) variables. If the first player playing on such $B_i$ did no play $a_i$, then we know Staller made the move on $b_i$ and Dominator replied by playing $e_i$. Thus by Dominator's strategy, he can play $a_i$ as the last move on $B_i$. If Dominator now finished playing the POS-CNF game with the remaining variables, selecting moves of Player 2 arbitrarily, he would still win, and by the argument above, it still holds that for all indices $i \in I$ that Player 1 sets to TRUE in the POS-CNF game on $\F$, the vertices $a_i$, $i \in I$, are played. Hence, in this case, all vertices $c_1, \ldots, c_n, \ldots, c_1^n, \ldots, c_n^n$ are dominated as well. Thus Dominator also ensures that by keeping at most three moves played on each $B_i$ and at most two on each $C_j$, all vertices from $G_\F - H_{2n+7} - A$ can be dominated (i.e.\ it cannot happen that no move on some $C_j$ would be legal). 
	
	If $k$ is even, there are no moves on $A$. If $k$ is odd and Staller made the first move on $A$, then by (1), three moves are played on $A$. Thus the only remaining case to consider is if $k$ is odd and Dominator is forced to play the first move on $A$. 	This happens if only vertices on $A$ are undominated, $p_1$ has not been played before, and it is Dominator's turn now, Dominator must play $p_1$ (it is the only legal move in the game), thus possibly four moves are made on $A$. (His consequent move is $p_2$ or $p_3$.) If at most $2n+8+3k+2n-1 = 4n+3k+7$ moves were made during the game before this move, the game ends after at most $4n+3k+11$ moves. Otherwise, so if exactly $2n+8+3k+2n = 4n+3k+8$ moves were made before (by arguments above Dominator ensures at most three moves on each $B_i$ and at most two on each $C_j$), as $k$ is odd, $4n+3k+8$ is also odd, so it cannot be Dominator's turn now. Thus if four moves are played on $A$, then at most $4n+3k+7$ moves were played on the rest of the graph.

	It follows from the above that the game ends in at most $2n+8+3k+2n = 4n+3k+8$ moves if $k$ is even, and in at most $2n+8+3k+2n+3=4n+3k+11$ moves if $k$ is odd.	
\end{proof}

\begin{lemma}
	\label{lem:p2-wins}
	If Player 2 has a winning strategy for the POS-CNF game played on $\F$, then $$\gcg(G_\F) \geq \begin{cases}
		3k + 4n + 9; & k \text{ even},\\
		3k + 4n + 12; & k \text{ odd}.
	\end{cases}$$
\end{lemma}

\begin{proof}
	We describe Staller's strategy that ensures that the game ends in at least $3k + 4n + 9$ moves if $k$ is even and in at least $3k + 4n + 12$ moves if $k$ is odd. Most of the proof is the same for both parities of $k$. Staller's strategy is composed of two phases and a series of general rules that Staller follows during both phases. Phase 1 lasts until too many vertices have been played on $H_{2n+7}$, $A$, some gadget $B_i$ or some gadget $C_j$. Afterwards, the game is in Phase 2 till the end. The general rules and both phases are described in detail later.
	
	If Dominator's first move is a vertex from $G_\F - H_{2n+7}$, then Staller's strategy is as follows. She partitions the graph into two subgraphs, $H_{2n+7} - u_0$ and $G_\F - (H_{2n+7} - u_0)$, and when Dominator plays on one of the parts, she plays on the same part if possible. Additionally, she follows the strategy \slow\ on $H_{2n+7}$, i.e.\ she plays on $x_1, \ldots, x_{2n+6}$ whenever possible. By Lemma \ref{lem:gc-min}, at least $3k + n + 1$ moves are played on $G_\F - (H_{2n+7} - u_0)$ (including $u_0$). As Staller is not able to reply to Dominator's move on $G_\F - (H_{2n+7} - u_0)$ at most once, she is able to ensure that at least $2n+5$ vertices among $x_1, \ldots, x_{2n+6}$ are played during the game. Thus the game will last for at least $(3k + n + 1) + (2n+7) + (2n+5) = 5n+3k+13$ which is more than the desired lower bound.
	
	If Dominator's first move is a vertex from $H_{2n+7} - u_{2n+7}$, at least $2n+9$ moves are needed to dominate $H_{2n+7}$. If Staller plays $u_0$, the rest of the game proceeds as in the next case (i.e.\ if Dominator started on $u_{2n+7}$), except that the game is in Phase 2 from the start. If Dominator plays $u_0$, then Staller plays $b_1$ in her next move. For the rest of the game (which is also in Phase 2 already), she uses the same strategy as if Dominator started on $u_{2n+7}$, and pretends that she has not played $b_1$ yet. Thus during the strategy described below, it may happen that Staller's prescribed move is not legal anymore. In this case she pretends that she played the move now, follows any additional part of her strategy related to this move, but in reality she plays \emph{any legal move} with the following restrictions: if possible she never plays vertices $c_1, \ldots, c_n$ or $a_1, \ldots, a_k$. Additionally, if she plays two moves on the same $B_i$ while Dominator plays no moves there, she plays $b_i$ and $f_i^1$, thus the game enters Phase 2.
	
	From now on, assume that $d_1 = u_{2n+7}$. The next few moves are forced: $s_1 = u_{2n+6}, d_2 = u_{2n+5}, \ldots, d_{2n+7} = u_1, s_{2n+8} = u_0$. As $2n+8$ is even, Staller plays $u_0$. Observe again that no vertex $d_j^m$ can be a legal move in the rest of the game.
	
	For the rest of the game, Staller will imagine a game is being played on $\F$, specifying certain moves of Dominator and herself as moves in the POS-CNF game. More precisely, in the game on $\F$, Player 2 will be playing optimally, and their moves will determine some of the corresponding moves of Staller on $G_\F$. As Player 2 has a winning strategy on $\F$, they can win no matter how Player 1 plays. Player 1's moves on $\F$ will be determined by some of Dominator's moves on $G_\F$. Which moves of Dominator mean that Player 1 makes a move on $\F$ and how does Player 2's reply on $\F$ translate to Staller's reply on $G_\F$ is explained in (2) below.
	
	 Staller's staretegy consists of two phases, but some part of her strategy is the same for both. Again we list rules with a concise but simplified description first.
	
	\begin{description}
		\item [General rules] If Dominator's move is among these rules, Staller plays according to them no matter which phase the game is in.
		\begin{enumerate}[(A)]
			\item If Dominator plays $p_1$, Staller plays $q_1$.\\
			This ensures at least four moves are played on $A$ if Dominator is the first player to play on it.
			\item If Dominator plays $p_2$, Staller plays $q_2$.
			\item If Dominator plays some $a_i$ in $B_i$, then Staller plays on appropriate gadgets $C_j$ for the next sequence of moves, until the game enters Phase 2 or she is directed to play on $B_i$.\\
			As $a_i$ was a legal move for Dominator, it newly dominated some $c_{j_1}, \ldots, c_{j_m}$, $m\geq 1$. Staller's strategy is to first play $c_{j_1}^1$. If Dominator replies by playing $c_{j_1}$, then Staller plays $c_{j_2}^1$. If Dominator replies to Staller's move on $c_{j_r}^1$ by playing $c_{j_r}$, then Staller plays $c_{j_{r+1}}^1$ next. If it exists, let $p \in [m]$ be the first index where Staller played $c_{j_p}^1$ and Dominator did not reply by playing $c_{j_p}$. If Dominator plays on $A$, Staller follows (A) or (B). If Dominator plays some other $a_{i'}$, she follows (C) from the start (each $C_{j'}$ on which moves can be played thus has an associated gadget $B_{i'}$ and set of gadgets $C_j$). Otherwise, Staller plays a vertex from $\{c_{j_p}^2, \ldots, c_{j_p}^n\}$. If Dominator is playing only on vertices $c_j$, then this strategy of Staller ensures that on average at least two moves are played on each gadget $C_j$ (only one on some, but accordingly more on $C_{j_p}$). If Dominator plays any other vertex, Staller's strategy ensures that at least three vertices will be played on $C_{j_p}$ (even after averaging the counting as before), thus the game enters Phase 2.
			
			If no move on $C_{j_1}, \ldots, C_{j_m}$ is legal and the game is still in Phase 1, then Staller plays on $B_i$ according to the rules given in (1) below (the game stays in Phase 1 only if Dominator is playing only on vertices $c_{j_1}, \ldots, c_{j_m}$ until no more move on $C_{j_1}, \ldots, C_{j_m}$ is legal, thus no player made a move on the rest of the graph in the meantime). If no move on $C_{j_1}, \ldots, C_{j_m}$ is legal and the game is now in Phase 2, then Staller follows the rules of Phase 2.
		\end{enumerate}

		\item[Phase 1] Exactly $2n+8$ moves were played on $H_{2n+7}$, at most three moves were played on $A$, and on average, at most three moves were played on every $B_i$, $i \in [k]$, and at most two moves on every $C_j$, $j \in [n]$.\\
		Staller plays according to the general rules in addition to the rules given below. By saying on average when counting moves we mean that possibly there are four moves on some $B_i$, but then this $B_i$ is associated with a gadget $C_j$ on which only one move was played, or that possibly there are more moves played on some $C_j$ but then there is only one move played on accordingly many other gadgets $C_{j'}$.
		\begin{enumerate}[(1)]
			\item If Dominator plays on $B_i$ and not all vertices of $B_i$ are dominated after his move, Staller plays on $B_i$.\\
			More precisely, if Dominator's move was the first on the gadget, then he played $a_i$ or $b_i$. If he played $a_i$, Staller follows the general rule (C) and if it eventually directs her to play on $B_i$, she plays $b_i$. If Dominator played $b_i$, she selects $f^1_i$. In the latter case, at least four moves will be needed to dominate $B_i$, thus the game enters Phase 2. See Figure~\ref{fig:St-strat-1a}. 
			
			\begin{figure}[!ht]
				\begin{center}
					\begin{tikzpicture}[thick, scale = 1]
						\tikzstyle{every node}=[circle, draw, fill=black!10, inner sep=0pt, minimum width=4pt]
						
						\begin{scope}			
							\node[fill=black] (a) at (0,0) {};
							\node[fill=black] (e) at (1,0) {};
							\node[fill=black] (b) at (2,0) {};
							\node[fill=black] (d) at (3,0) {};
							
							\node[fill=black] (h) at (1,1) {};
							\node[fill=black] (k1) at (0.5,1) {};
							
							\node[fill=black] (f1) at (2,1) {};
							\node (g1) at (1.5,1) {};
							
							\node[fill=black] (f2) at (2,-1) {};
							\node (g2) at (1.5,-1) {};
							
							\draw (a) -- (e) -- (b) -- (d);
							\draw (a) -- (h) -- (e) -- (f1) -- (b) -- (f2) -- (e) -- (g1) -- (f1);
							\draw (e) -- (g2) -- (f2);
							\draw (h) -- (k1) -- (a);
							
							\node[fill=none, draw=darkgray,
							inner sep=1pt, minimum width=10pt,line width=2pt,label=below:{$d_j$}] (d1) at (0,0) {};
							\node[fill=none, draw=gray,
							inner sep=1pt, minimum width=10pt,line width=2pt,label=below right:{$s_{j'}$}] (s1) at (2,0) {};
						\end{scope}
						
						\begin{scope}[xshift=5cm]		
							\node (a) at (0,0) {};
							\node[fill=black] (e) at (1,0) {};
							\node[fill=black] (b) at (2,0) {};
							\node[fill=black] (d) at (3,0) {};
							
							\node (h) at (1,1) {};
							\node (k1) at (0.5,1) {};
							
							\node[fill=black] (f1) at (2,1) {};
							\node[fill=black] (g1) at (1.5,1) {};
							
							\node[fill=black] (f2) at (2,-1) {};
							\node (g2) at (1.5,-1) {};

							\draw (a) -- (e) -- (b) -- (d);
							\draw (a) -- (h) -- (e) -- (f1) -- (b) -- (f2) -- (e) -- (g1) -- (f1);
							\draw (e) -- (g2) -- (f2);
							\draw (h) -- (k1) -- (a);
							
							\node[fill=none, draw=darkgray,
							inner sep=1pt, minimum width=10pt,line width=2pt,label=below right:{$d_j$}] (d1) at (2,0) {};
							\node[fill=none, draw=gray,
							inner sep=1pt, minimum width=10pt,line width=2pt,label=right:{$s_j$}] (s1) at (2,1) {};
						\end{scope}

					\end{tikzpicture}
					\caption{Staller's strategy in (1) if Dominator played first on $B_i$. Dominated vertices are colored black.}
					\label{fig:St-strat-1a}
				\end{center}
			\end{figure}
			
			If Dominator's move was not the first move on $B_i$, then Staller made the first move on $B_i$ by playing $b_i$ (by (2) and her restrictions when playing any legal move). If Dominator plays $e_i$, Staller replies on $h_i$. If he plays $f_i^{\ell}$ for some $\ell \in \{1,2\}$, then she plays $f_i^{3-\ell}$. Observe that in this case, at least one additional move is needed to dominate $B_i$, thus the game already enters Phase 2. If he plays $a_i$, she follows the general rule (C); if the rule eventually directs her to play on $B_i$, she selects $f_i^1$. In this case, at least one more move is needed on $B_i$, thus the game also enters Phase 2. See Figure~\ref{fig:St-strat-1b}.
			
			\begin{figure}[!ht]
				\begin{center}
					\begin{tikzpicture}[thick, scale = 1]
						\tikzstyle{every node}=[circle, draw, fill=black!10, inner sep=0pt, minimum width=4pt]
						
						\begin{scope}			
							\node[fill=black] (a) at (0,0) {};
							\node[fill=black] (e) at (1,0) {};
							\node[fill=black] (b) at (2,0) {};
							\node[fill=black] (d) at (3,0) {};
							
							\node[fill=black] (h) at (1,1) {};
							\node[fill=black] (k1) at (0.5,1) {};
							
							\node[fill=black] (f1) at (2,1) {};
							\node[fill=black] (g1) at (1.5,1) {};
							
							\node[fill=black] (f2) at (2,-1) {};
							\node[fill=black] (g2) at (1.5,-1) {};

							\draw (a) -- (e) -- (b) -- (d);
							\draw (a) -- (h) -- (e) -- (f1) -- (b) -- (f2) -- (e) -- (g1) -- (f1);
							\draw (e) -- (g2) -- (f2);
							\draw (h) -- (k1) -- (a);
							
							\node[fill=none, draw=darkgray,
							inner sep=1pt, minimum width=10pt,line width=2pt,label=below:{$d_j$}] (d1) at (1,0) {};
							\node[fill=none, draw=gray,
							inner sep=1pt, minimum width=10pt,line width=2pt,label=below right:{$s_{j_0}$}] (s1) at (2,0) {};
							\node[fill=none, draw=gray,
							inner sep=1pt, minimum width=10pt,line width=2pt,label=above:{$s_j$}] (s2) at (1,1) {};
						\end{scope}
						
						\begin{scope}[xshift=5cm]		
							\node[fill=black] (a) at (0,0) {};
							\node[fill=black] (e) at (1,0) {};
							\node[fill=black] (b) at (2,0) {};
							\node[fill=black] (d) at (3,0) {};
							
							\node (h) at (1,1) {};
							\node (k1) at (0.5,1) {};
							
							\node[fill=black] (f1) at (2,1) {};
							\node[fill=black] (g1) at (1.5,1) {};
							
							\node[fill=black] (f2) at (2,-1) {};
							\node[fill=black] (g2) at (1.5,-1) {};

							\draw (a) -- (e) -- (b) -- (d);
							\draw (a) -- (h) -- (e) -- (f1) -- (b) -- (f2) -- (e) -- (g1) -- (f1);
							\draw (e) -- (g2) -- (f2);
							\draw (h) -- (k1) -- (a);
							
							\node[fill=none, draw=darkgray,
							inner sep=1pt, minimum width=10pt,line width=2pt,label=above:{$d_j$}] (d1) at (2,1) {};
							\node[fill=none, draw=gray,
							inner sep=1pt, minimum width=10pt,line width=2pt,label=below right:{$s_{j_0}$}] (s1) at (2,0) {};
							\node[fill=none, draw=gray,
							inner sep=1pt, minimum width=10pt,line width=2pt,label=below:{$s_{j}$}] (s2) at (2,-1) {};
						\end{scope}
					
						\begin{scope}[xshift=10cm]		
							\node[fill=black] (a) at (0,0) {};
							\node[fill=black] (e) at (1,0) {};
							\node[fill=black] (b) at (2,0) {};
							\node[fill=black] (d) at (3,0) {};
							
							\node[fill=black] (h) at (1,1) {};
							\node[fill=black] (k1) at (0.5,1) {};
							
							\node[fill=black] (f1) at (2,1) {};
							\node[fill=black] (g1) at (1.5,1) {};
							
							\node[fill=black] (f2) at (2,-1) {};
							\node (g2) at (1.5,-1) {};

							\draw (a) -- (e) -- (b) -- (d);
							\draw (a) -- (h) -- (e) -- (f1) -- (b) -- (f2) -- (e) -- (g1) -- (f1);
							\draw (e) -- (g2) -- (f2);
							\draw (h) -- (k1) -- (a);
							
							\node[fill=none, draw=darkgray,
							inner sep=1pt, minimum width=10pt,line width=2pt,label=below:{$d_j$}] (d1) at (0,0) {};
							\node[fill=none, draw=gray,
							inner sep=1pt, minimum width=10pt,line width=2pt,label=below right:{$s_{j_0}$}] (s1) at (2,0) {};
							\node[fill=none, draw=gray,
							inner sep=1pt, minimum width=10pt,line width=2pt,label=above:{$s_{j'}$}] (s2) at (2,1) {};
						\end{scope}	
						
					\end{tikzpicture}
					\caption{Staller's strategy in (1) if Dominator played second on $B_i$. Dominated vertices are colored black.}
					\label{fig:St-strat-1b}
				\end{center}
			\end{figure}
			
			Notice that if Dominator makes the first move on $B_i$, then exactly three moves will be played on that $B_i$ or the game will enter Phase 2. If Staller makes the first move on $B_i$, then either three moves are played on it and $a_i$ is not one of them, or the game enters Phase 2.
			
			\item If Dominator plays on $B_i$, all vertices of $B_i$ are dominated after his move and the POS-CNF game has not ended yet, Staller selects an appropriate gadget $B_j$ to play on.\\
			If $X_i$ already has a value assigned in the POS-CNF game, then Staller plays any legal move. Otherwise, Staller sets $X_i$ to TRUE as the move of Player 1 in the POS-CNF game. If this ends the POS-CNF game, Staller plays as in (3). Otherwise, she lets Player 2 play optimally on $\F$, setting some $X_j$ to FALSE. Staller now tries to play on $B_j$.
			If no move was made before on $B_j$, Staller plays $b_j$. If exactly one move was made on $B_j$ before we know by (1) that it was a move of Staller, so she played $b_j$, and can play $f_j^1$ now, thus the game enters Phase 2. If two moves were played on $B_j$ so far, then as we are in Phase 1 and by (1), we know that Dominator played $a_j$ and Staller played $b_j$. Now Staller plays $f_j^1$, thus the game enters Phase 2.	It follows from the above that more moves on $B_j$ could not have been played yet as we are in Phase 1.		
			
			

			Note that this means that while we are in Phase 1, a variable $X_j$ is set to FALSE only if Staller played first on $B_j$, $a_j$ has not been played, and at most three moves were played on $B_j$. As long as the game is in Phase 1, we also know that $a_j$ will not be played.
			
%
			
			\item Otherwise, so if Dominator played on $B_i$, all vertices of $B_i$ are dominated after his move and the POS-CNF game has ended already, Staller does the following.\\
			The POS-CNF game has ended with the win of Player 2. If not all gadgets $B_i$ are dominated yet, Staller plays on one of them, but not playing $a_i$ (which she can as if $a_i$ is legal, so is $h_i$). If all gadgets $B_i$ are dominated already, then as Player 2 won the POS-CNF game, we know that there is at least one $C_j$ that is not dominated yet. As we are in Phase 1, we know that no move has been made on $A$ yet (by (A)). As all $B_i$ are dominated and we are in Phase 1, exactly three moves were made on each gadget $B_i$ and exactly zero or two moves were made on each gadget $C_j$ (on average). If $k$ is even, an even number of moves was played in the game so far, so it cannot be Staller's turn now. If $k$ is odd, it can indeed be Staller's turn now, and she plays $p_1$ (which is present as $k$ is odd). If Dominator replies by playing $q_1$, the game enters Phase 2. Otherwise, Dominator plays $p_2$, Staller plays $p_3$, and now it is Dominator's turn. The only legal moves for him are to play some $a_i$ that newly dominates some $c_j$. Staller now plays $c_j^1$ (as in (C)). This ensures four moves on $B_i$ and at least two on $C_j$, thus the game enters Phase 2.
		\end{enumerate}
		
		\item[Phase 2] At least $2n+9$ moves were played on $H_{2n+7}$, four moves have been or will be played on $A$, or on average, at least four moves have been or will be played on some gadget $B_i$ or at least three moves have been or will be played on some gadget $C_j$.\\
		Staller plays according to the general rules or if they are not applicable, she plays any legal move until the game ends with the following restrictions.
		\begin{itemize}
			\item If possible, Staller never plays vertices $c_1, \ldots, c_n$.\\
			Staller is forced to play on $c_1, \ldots, c_n$ only if some of these are the only legal moves in the game. But as $c_j$ is the only legal move in $C_j$ only if $d_j$ is the only undominated vertex in $C_j$, and so if $c_j^1, \ldots, c_j^n$ have all been played already, this move maintains the property that at least two moves are played on each $C_j$. 
			
			\item If possible, Staller never plays vertices $a_1, \ldots, a_k$.\\
			Suppose that $a_i$ is a legal move for Staller. If not all vertices of $B_i$ are already dominated, she can play $h_i$ instead (which is also legal as $a_i$ is legal). Otherwise, so if $a_i$ is a legal move only because it newly dominates vertices of some $C_j$, observe that at least three moves were already played on $B_i$ as it is already dominated. Staller now plays $a_i$ as the fourth move on $B_i$, allowing Dominator to dominate $C_j$ with only one move. But we can still count this as three moves on $B_i$ and two moves on $C_j$.
		\end{itemize}
	\end{description}
	
	Notice that by above rules, Staller is always ensuring that at least three moves are played on each $B_i$ and at least two on each $C_j$, on average. It also follows from the rules that the game always enters Phase 2. Thus Staller's strategy ensures that at least $(2n+8) + 3k + 2n + 1 = 4n+3k + 9$ moves are made if $k$ is even and at least $(2n+8) + 3k + 2n + 4 = 4n+3k + 12$ moves are made if $k$ is odd.
\end{proof}

We remark that a more detailed analysis of the game might yield that the graph $H_{2n+7}$ in $G_\F$ could be replaced with some $H_i$, $i \leq 2n+6$, but feel that a simpler proof better illustrates the general idea.

\begin{theorem}
	\label{thm:main}
	The \textsc{Game connected domination problem} is PSPACE-complete.
\end{theorem}

\begin{proof}
	It is easy to see that the \textsc{Game connected domination problem} is in PSPACE. (For example, use the fact that NP $\subseteq$ PSPACE.)
	
	To prove that the problem is PSPACE-complete, we use a reduction from the \textsc{POS-CNF problem}. Given a POS-CNF formula $\F$, the graph $G_\F$ is obtained as described in Section \ref{sec:construction}. Combining Lemmas \ref{lem:p1-wins} and \ref{lem:p2-wins} implies the following. If $k$ is even, $\gcg(G_\F) \leq 3k + 4n + 8$ if and only if Player 1 wins the POS-CNF game on $\F$. If $k$ is odd, $\gcg(G_\F) \leq 3k + 4n + 11$ if and only if Player 1 wins the POS-CNF game on $\F$. Since POS-CNF problem is known to be PSPACE-complete and the described reduction from \textsc{POS-CNF problem} to \textsc{Game connected domination problem} can be computed with a polynomial size working space, the desired result follows.
\end{proof}

Since the reduction in Theorem~\ref{thm:main} can be computed with a logarithmic size working space, and since \textsc{POS-CNF problem} is known to be log-complete in PSPACE~\cite{pos-cnf}, we obtain an even stronger result.

\begin{corollary}
	\label{cor:main}
	The \textsc{Game connected domination problem} is log-complete in PSPACE.
\end{corollary}

\section{Complexity of the Staller-start game}
\label{sec:S-game}

In this section we consider the Staller-start connected domination game. The main idea of the proof of the PSPACE-completeness of the \textsc{Staller-start game connected domination problem} is similar as in the Dominator-start game. However, the graph $H_{2n+7}$ in the construction is replaced with a smaller graph and connected to the rest of the graph in a slightly different way.

Given a formula $\F$ with $k$ variables and $n$ disjunctive clauses, we built a graph $G'_{\F}$ in the following way. For each variable $X_i$, $i \in [k]$, we add to the graph a copy $B_i$ of the graph $B$ (called a gadget $B_i$). For each clause $C_j$, $j \in [n]$, we add to the graph a copy $C(n)_j = C_j$ of the graph $C(n)$ (called a gadget $C_j$) and make all vertices $\{c_j, c^1_j, \ldots, c^n_j\}$ adjacent to $a_i$ if and only if the variable $X_i$ appears in the clause $C_j$. We add to the graph a disjoint copy of the graph $H_6$ and make $u_7$ adjacent to vertices $a_i$ and $b_i$ for all $i \in [k]$.  If $k$ is odd, we also add to the graph a disjoint copy of the graph $A$ and make $p_1$ adjacent to $u_7$. For example, see Figure \ref{fig:GF'}.

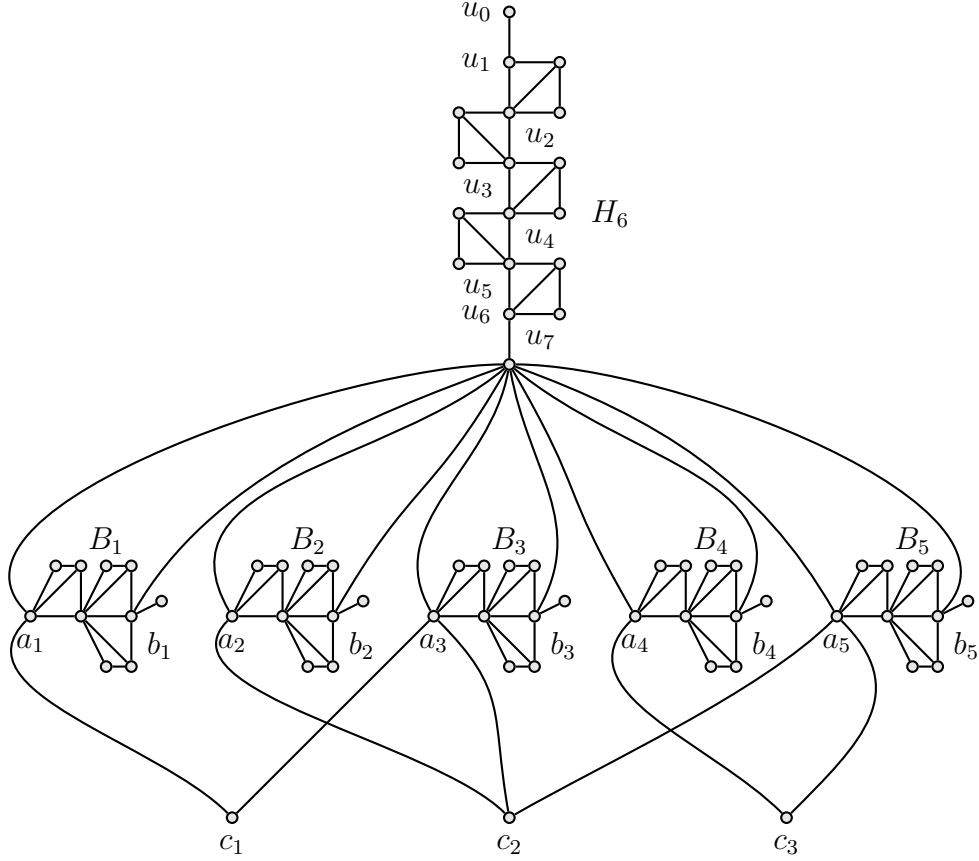
\begin{figure}[!ht]
	\begin{center}
		\begin{tikzpicture}[thick, scale = 0.77]
			\tikzstyle{vert}=[circle, draw, fill=black!10, inner sep=0pt, minimum width=4pt]
			
			\foreach \x in {1,2,3,4,5}{
				\node (k\x) at (4*\x - 4 +1.5,1.5) {$B_{\x}$};
				\node[vert, label=below: {$a_{\x}$}] (a\x) at (4*\x - 4 +0,0) {};
				\node[vert] (e\x) at (4*\x - 4+1,0) {};
				\node[vert, label=below right: {$b_{\x}$}] (b\x) at (4*\x - 4+2,0) {};
				\node[vert] (d\x) at (4*\x - 4+2.6,0.3) {};
				
				\node[vert] (h\x) at (4*\x - 4+1,1) {};
				\node[vert] (k1\x) at (4*\x - 4+0.5,1) {};
				
				\node[vert] (f1\x) at (4*\x - 4+2,1) {};
				\node[vert] (g1\x) at (4*\x - 4+1.5,1) {};
				
				\node[vert] (f2\x) at (4*\x - 4+2,-1) {};
				\node[vert] (g2\x) at (4*\x - 4+1.5,-1) {};

				\draw (a\x) -- (e\x) -- (b\x) -- (d\x);
				\draw (a\x) -- (h\x) -- (e\x) -- (f1\x) -- (b\x) -- (f2\x) -- (e\x) -- (g1\x) -- (f1\x);
				\draw (e\x) -- (g2\x) -- (f2\x);
				\draw (h\x) -- (k1\x) -- (a\x);	
			}
			
			\node (h) at (12.5,8.5) {$H_{6}$};
			\node[vert, label=5: {$u_7$}] (u0) at (9.5,5) {};
			\node[vert, label=180: {$u_6$}] (u1) at (9.5,7) {};
			\node[vert] (u2) at (10.5,7) {};
			\node[vert] (u3) at (11.5,7) {};
			\node[vert] (u4) at (12.5,7) {};
			\node[vert] (u5) at (13.5,7) {};
			\node[vert, label=-90: {$u_1$}] (u6) at (14.5,7) {};
			\node[vert, label=-90: {$u_0$}] (u7) at (15.5,7) {};
			
			\node[vert] (x1) at (9.5,8) {};
			\node[vert] (x2) at (10.5,6) {};
			\node[vert] (x3) at (11.5,8) {};
			\node[vert] (x4) at (12.5,6) {};
			\node[vert] (x5) at (13.5,8) {};
			
			\node[vert] (y1) at (10.5,8) {};
			\node[vert] (y2) at (11.5,6) {};
			\node[vert] (y3) at (12.5,8) {};
			\node[vert] (y4) at (13.5,6) {};
			\node[vert] (y5) at (14.5,8) {};
			
			\draw (u0) -- (u1) -- (u2) -- (u3) -- (u4) -- (u5) -- (u6) -- (u7);
			
			\draw (u1) -- (x1) -- (y1);
			\path (u1) edge (y1);
			\path (u2) edge (y1);
			
			\draw (u2) -- (x2) -- (y2);
			\path (u2) edge (y2);
			\path (u3) edge (y2);
			
			\draw (u3) -- (x3) -- (y3);
			\path (u3) edge (y3);
			\path (u4) edge (y3);
			
			\draw (u4) -- (x4) -- (y4);
			\path (u4) edge (y4);
			\path (u5) edge (y4);
			
			\draw (u5) -- (x5) -- (y5);
			\path (u5) edge (y5);
			\path (u6) edge (y5);
			
			\path (u6) edge (u7);
			
			\draw (u0) to[out=180,in=135,distance=3cm] (a1);
			\draw (u0) to[out=-160,in=60,distance=3cm] (b1);
			\draw (u0) to[out=-140,in=120,distance=3cm] (a2);
			\draw (u0) to[out=-120,in=60,distance=3cm] (b2);
			
			\draw (u0) to[out=-100,in=120,distance=2cm] (a3);
			\draw (u0) to[out=-80,in=60,distance=2cm] (b3);
			
			\draw (u0) to[out=-60,in=120,distance=3cm] (a4);
			\draw (u0) to[out=-40,in=60,distance=3cm] (b4);
			\draw (u0) to[out=-20,in=120,distance=3cm] (a5);
			\draw (u0) to[out=0,in=45,distance=3cm] (b5);
			
			\node (a) at (2.5,7) {$A$ as $k$ is odd};		
			\node[vert, label=45: {$p_1$}] (p1) at (6.5,7) {};
			\node[vert, label=-45: {$p_2$}] (p2) at (5.5,7) {};
			\node[vert, label=45: {$p_3$}] (p3) at (4.5,7) {};
			
			\node[vert, label=above: {$q_1$}] (q1) at (6.5,8) {};
			\node[vert, label=below: {$q_2$}] (q2) at (5.5,6) {};
			
			\node[vert, label=above: {$r_1$}] (r1) at (5.5,8) {};
			\node[vert, label=below: {$r_2$}] (r2) at (4.5,6) {};
			
			\draw (p1) -- (p2) -- (p3);
			
			\draw (p1) -- (q1) -- (r1);
			\path (p2) edge (q1);
			\path (p2) edge (r1);
			
			\draw (p2) -- (q2) -- (r2);
			\path (p3) edge (q2);
			\path (p3) edge (r2);
			
			\draw (u0) -- (p1);
			
			\node (cc1) at (2,-6.3) {$C_1$};
			\node[vert, label=above: {$c_1$}] (c1) at (0.5,-4) {};
			\node[vert, label=right: {$c^1_1$}] (c11) at (1.5,-4) {};
			\node[vert, label=right: {$c^2_1$}] (c21) at (2.5,-4) {};
			\node[vert, label=right: {$c^3_1$}] (c31) at (3.5,-4) {};
			
			\node[vert, label=below: {$d_1$}] (d1) at (0.5,-5) {};
			\node[vert, label=below: {$d^1_1$}] (d11) at (1.5,-5) {};
			\node[vert, label=below: {$d^2_1$}] (d21) at (2.5,-5) {};
			\node[vert, label=below: {$d^3_1$}] (d31) at (3.5,-5) {};
			
			\draw (c1) -- (d1);
			\draw (c11) -- (d11);
			\draw (c21) -- (d21);
			\draw (c31) -- (d31);
			
			\draw (c1) -- (d11);	
			\draw (c1) -- (d21);
			\draw (c1) -- (d31);
			
			\node (cc2) at (8.5,-6.3) {$C_2$};
			\node[vert, label=above: {$c_2$}] (c2) at (7,-4) {};
			\node[vert, label=right: {$c^1_2$}] (c12) at (8,-4) {};
			\node[vert, label=right: {$c^2_2$}] (c22) at (9,-4) {};
			\node[vert, label=right: {$c^3_2$}] (c32) at (10,-4) {};
			
			\node[vert, label=below: {$d_2$}] (d2) at (7,-5) {};
			\node[vert, label=below: {$d^1_2$}] (d12) at (8,-5) {};
			\node[vert, label=below: {$d^2_2$}] (d22) at (9,-5) {};
			\node[vert, label=below: {$d^3_2$}] (d32) at (10,-5) {};
			
			\draw (c2) -- (d2);
			\draw (c12) -- (d12);
			\draw (c22) -- (d22);
			\draw (c32) -- (d32);
			
			\draw (c2) -- (d12);	
			\draw (c2) -- (d22);
			\draw (c2) -- (d32);
			
			\node (cc3) at (15,-6.3) {$C_3$};
			\node[vert, label=above: {$c_3$}] (c3) at (13.5,-4) {};
			\node[vert, label=right: {$c^1_3$}] (c13) at (14.5,-4) {};
			\node[vert, label=right: {$c^2_3$}] (c23) at (15.5,-4) {};
			\node[vert, label=right: {$c^3_3$}] (c33) at (16.5,-4) {};
			
			\node[vert, label=below: {$d_3$}] (d3) at (13.5,-5) {};
			\node[vert, label=below: {$d^1_3$}] (d13) at (14.5,-5) {};
			\node[vert, label=below: {$d^2_3$}] (d23) at (15.5,-5) {};
			\node[vert, label=below: {$d^3_3$}] (d33) at (16.5,-5) {};
			
			\draw (c3) -- (d3);
			\draw (c13) -- (d13);
			\draw (c23) -- (d23);
			\draw (c33) -- (d33);
			
			\draw (c3) -- (d13);	
			\draw (c3) -- (d23);
			\draw (c3) -- (d33);
			
			\draw (c1) to[out=135,in=-135,distance=2cm] (a1);
			\draw (c11) to[out=135,in=-135,distance=2cm] (a1);
			\draw (c21) to[out=130,in=-135,distance=2cm] (a1);
			\draw (c31) to[out=130,in=-135,distance=2cm] (a1);
			
			\draw (c1) to[out=55,in=-135,distance=3cm] (a3);
			\draw (c11) to[out=60,in=-135,distance=3cm] (a3);
			\draw (c21) to[out=65,in=-135,distance=3cm] (a3);
			\draw (c31) to[out=70,in=-135,distance=3cm] (a3);
			
			\draw (c2) to[out=135,in=-145,distance=2.5cm] (a2);
			\draw (c12) to[out=135,in=-145,distance=2.5cm] (a2);
			\draw (c22) to[out=130,in=-145,distance=2.5cm] (a2);
			\draw (c32) to[out=130,in=-145,distance=2.5cm] (a2);
			
			\draw (c2) to[out=50,in=-45,distance=1.3cm] (a3);
			\draw (c12) to[out=60,in=-45,distance=1.3cm] (a3);
			\draw (c22) to[out=70,in=-45,distance=1.3cm] (a3);
			\draw (c32) to[out=80,in=-45,distance=1.3cm] (a3);
			
			\draw (c2) to[out=45,in=-135,distance=4cm] (a5);
			\draw (c12) to[out=45,in=-135,distance=4cm] (a5);
			\draw (c22) to[out=45,in=-135,distance=4cm] (a5);
			\draw (c32) to[out=45,in=-135,distance=4cm] (a5);
			
			\draw (c3) to[out=135,in=-135,distance=2.5cm] (a4);
			\draw (c13) to[out=135,in=-135,distance=2.5cm] (a4);
			\draw (c23) to[out=130,in=-135,distance=2.5cm] (a4);
			\draw (c33) to[out=130,in=-135,distance=2.5cm] (a4);
			
			\draw (c3) to[out=45,in=-45,distance=2cm] (a5);
			\draw (c13) to[out=55,in=-45,distance=2cm] (a5);
			\draw (c23) to[out=65,in=-45,distance=2cm] (a5);
			\draw (c33) to[out=75,in=-45,distance=2cm] (a5);
		\end{tikzpicture}
		\caption{The graph $G'_\F$ obtained from the formula $\F = (X_1 \vee X_3) \wedge (X_2 \vee X_3 \vee X_5) \wedge (X_4 \vee X_5)$.}
		\label{fig:GF'}
	\end{center}
\end{figure}

\begin{lemma}
	\label{lem:p1-wins'}
	If Player 1 has a winning strategy for the POS-CNF game played on $\F$, then $$\gcg'(G'_\F) \leq \begin{cases}
		3k + 2n + 13; & k \text{ even},\\
		3k + 2n + 16; & k \text{ odd}.
	\end{cases}$$
\end{lemma}

\begin{proof}
	Dominator's strategy consists of two phases. Phase 1 is the first phase of the game and Dominator's goal is to reach Phase 2 with as few moves as possible. Phase 2 starts with the first move of Dominator after the vertex $u_7$ has been played (so either with the move $d_i'$ if $s_i' = u_7$ or with the move $d_{i+1}'$ if $d_i' = u_7$). In Phase 2, Dominator follows the same strategy as in Lemma~\ref{lem:p1-wins} (after Staller played $u_0$ on $G_\F$). If during this strategy, if some vertex is not a legal move (because it was played before), Dominator considers the imagined POS-CNF game as if the move was played in the current step, but plays an arbitrary legal move in $G'_\F$, again following the rules for any legal move from the proof of Lemma~\ref{lem:p1-wins}. By the Connected Game Continuation Principle from~\cite{bujtas+2021connected} this is not a loss for Dominator. However, we may need to take these additional moves into account at the final count of moves. For this sake, let $$M = \begin{cases}
	 \{ s_1', d_1', \ldots, s_{i-1}', d_{i-1}' \} \cap \left( V(G'_\F) \setminus V(H_6) \right); & s_i' = u_7,\\
		\left( \{ s_1', d_1', \ldots, s_{i-1}', d_{i-1}' \} \cup \{ s_i', s_{i+1}' \} \right) \cap \left( V(G'_\F) \setminus V(H_6) \right); & d_i' = u_7,\\
	\end{cases}$$
	denote the set of possibly additional moves in $V(G'_{\F}) \setminus V(H_6)$ made during the game.
	If $s_1' \in V(H_6)$, then $|M| \leq 1$. If $s_1' \in V(G'_{\F}) \setminus V(H_6 \cup A)$, then since every vertex $x \in  V(G'_{\F}) \setminus V(H_6 \cup A)$ is at distance at most 3 from $u_7$, Dominator can ensure that $|M| \leq 6$ by playing the vertices on the shortest path between $s_1'$ and $u_7$. If $s_1' \in V(A)$, then it is not hard to see that $|M| \leq 6$ as well (only the vertex $r_2$ is at distance 4 from $u_7$ and has to be considered separately).
	
	To finalize the proof, consider the following cases. Let $$f(n,k) = \begin{cases}
		3k + 2n; & k \text{ even},\\
		3k + 2n + 3; & k \text{ odd}.
	\end{cases}$$ Observe that this equals the upper bound given in Lemma \ref{lem:p1-wins} without the $2n+8$ moves made on $H_{2n+7}$, so the moves played in the game after $u_0$ is played by Staller. 
	\begin{description}
		\item[Case 1.] $s_1' = u_0$.\\
		Dominator's strategy is to play only vertices from $\{u_1, \ldots, u_7\}$, thus at most 13 moves are played on $H_6$. If indeed 13 moves are played on $H_6$, then Staller plays $u_7$ and $M = \emptyset$. However, if Dominator is forced to play $u_7$, then at most 12 moves were made on $H_6$ and $|M| \leq 1$. In Phase 2, Dominator can ensure at most $f(n,k)$ moves are played on $V(G'_\F) \setminus V(H_6)$ since Player 1 has a winning strategy in the POS-CNF game on $\F$, using the same strategy as in the proof of Lemma \ref{lem:p1-wins}. Altogether, either at most $13 + f(n,k)$ moves are played, or at most $12 + f(n,k) + |M| \leq 13 + f(n,k)$ moves are made in the game.
		
		\item[Case 2.] $s_1' \in V(H_6) \setminus \{u_0\}$.\\
		Dominator's strategy is to play only vertices from $\{u_1, \ldots, u_7\}$, thus at most 12 moves are played on $H_6$. Combining Dominator's strategy from Phase 2 and the fact that $|M| \leq 1$, we conclude that at most $13 + f(n,k)$ moves are played in the game.
		
		\item[Case 3.] $s_1' \in V(G'_\F) \setminus V(H_6)$.\\
		Dominator's strategy is to ensure that $u_7$ is played as soon as possible. Afterwards, he utilizes the strategy \textsc{Fast} on $H_6$. More precisely, whenever Staller plays a vertex on $H_6$, Dominator replies by playing a vertex on $H_6$ as well. Since the remaining moves on $H_6$ can be paired ($(u_6, u_5), (u_4, u_3), (u_2, u_1)$) the development of the game on $H_6$ is independent of Dominator's strategy on $V(G'_\F) \setminus V(H_6)$, and exactly seven moves are played on $H_6$. Combining this with Dominator's strategy from Phase 2 and the fact that $|M| \leq 6$ yields that the number of moves in the game is at most $13 + f(n,k)$.
	\end{description}
	This concludes the proof of Lemma~\ref{lem:p1-wins'}.
\end{proof}

We remark that a more detailed analysis of the game might yield that the graph $H_6$ in $G'_\F$ could be replaced with some $H_i$, $i \leq 5$, but feel that a simpler proof better illustrates the general idea.

\begin{lemma}
	\label{lem:p2-wins'}
	If Player 2 has a winning strategy for the POS-CNF game played on $\F$, then $$\gcg'(G'_\F) \geq \begin{cases}
		3k + 2n + 14; & k \text{ even},\\
		3k + 2n + 17; & k \text{ odd}.
	\end{cases}$$
\end{lemma}

\begin{proof}
	It suffices to provide a strategy for Staller that ensures that at least $14 + f(n,k)$ moves are played on $G'_\F$ where $$f(n,k) = \begin{cases}
		3k + 2n; & k \text{ even},\\
		3k + 2n + 3; & k \text{ odd}.
	\end{cases}$$ Staller's strategy is to start the game by playing $u_0$ and using her strategy \textsc{Slow} on $H_6$. This ensures that 13 moves are played on $H_6$ and that Staller plays $u_7$. Now, using the same argument as in the proof of Lemma~\ref{lem:p2-wins} (after Staller played $u_0$) we can see that since Player 2 wins the POS-CNF game on $\F$, Staller can force at least $f(n,k)+1$ moves on $V(G'_\F) \setminus V(H_6)$. Thus the game lasts at least $14 + f(n,k)$ moves.
\end{proof}

Using analogous arguments as for the Dominator-start game, we obtain the following.

\begin{theorem}
	\label{thm:main'}
	The \textsc{Staller-start game connected domination problem} is PSPACE-complete. Moreover, the problem is log-complete in PSPACE.
\end{theorem}

\section*{Acknowledgments}
I would like to thank B.~Brešar, Cs.~Bujtás, S.~Cabello, S.~Klavžar and M.~Konvalinka for pointing out the problem studied in this paper, and acknowledge the financial support from the Slovenian Research and Innovation Agency (ARIS) under the grants Z1-50003, P1-0297, N1-0218, N1-0285 and N1-0355, and from the European Union (ERC, KARST, 101071836). I would also like to thank the anonymous reviewers for their close reading and helpful suggestions.

\printbibliography

\end{document}